\documentclass[a4paper,11pt]{article}
\usepackage[dvipsnames]{xcolor}
\usepackage{geometry}
\usepackage[english]{babel}
\usepackage{amssymb, enumerate}
\usepackage{latexsym}
\usepackage[all,cmtip]{xy}
\usepackage{amsmath,amssymb,amsthm}
\usepackage{comment}
\usepackage{tikz}
\usepackage{rotating}
\usepackage{marvosym}
\usepackage{adjustbox}

\usepackage{tabu}
\usepackage[colorlinks=true, citecolor=Blue, linkcolor=Blue, urlcolor=Blue, linktocpage=true, pdfpagelabels, ]{hyperref}
\usepackage{cleveref}

\title{$2$-blocks with an abelian defect group and a freely acting cyclic inertial quotient\thanks{The first author was supported by the London Mathematical Society (grant ECF-1920-03).}}
\author{
   Cesare Giulio Ardito\thanks{Department of Mathematics, City University of London, Northampton Square, London, EC1V 0HB, United Kingdom. Email: cesare.ardito@city.ac.uk}
\and 
Elliot McKernon \thanks{School of Mathematics, University of Manchester, Manchester, M13 9PL, United Kingdom. Email:
emckernon@gmail.com}}
\date{}

\newtheorem{theorem}{Theorem}[section]
\newtheorem*{theorem*}{Theorem}

\newtheorem{lemma}[theorem]{Lemma}
\newtheorem{corollary}[theorem]{Corollary}
\newtheorem{proposition}[theorem]{Proposition}

\newcommand{\Aut}{\operatorname{Aut}}
\newcommand{\Out}{\operatorname{Out}}
\newcommand{\aut}{\operatorname{Aut}}

\newcommand{\Pic}{\mathop{\rm Pic}\nolimits}

\newcommand{\SL}{\operatorname{SL}}

\newcommand{\GL}{\operatorname{GL}}

\newcommand{\cO} {\mathcal{O}}

\def\bigcp{\mathop{\mathchoice 
 {\hbox{\sf\Large\lower 0.1\baselineskip\hbox{Y}}}%
 {\hbox{\sf\large\lower 0.1\baselineskip\hbox{Y}}}%
 {\hbox{\sf\normalsize\lower 0.1\baselineskip\hbox{Y}}}%
 {\hbox{\sf\tiny\lower 0.1\baselineskip\hbox{Y}}}%
}}
\def\bigtimes{\mathop{\mathchoice 
 {\hbox{\sf\Large\lower 0.1\baselineskip\hbox{X}}}%
 {\hbox{\sf\large\lower 0.1\baselineskip\hbox{X}}}%
 {\hbox{\sf\normalsize\lower 0.1\baselineskip\hbox{X}}}%
 {\hbox{\sf\tiny\lower 0.1\baselineskip\hbox{X}}}%
}}

\def\Sym(#1){\mathop{\rm Sym}(#1)}
\def\Sym(#1){S_{#1}}
\def\diag(#1){\mathop{\rm diag}(#1)}

\begin{document}\setcounter{MaxMatrixCols}{50}
\maketitle 
\vspace{-1cm}
\begin{abstract}
We study blocks with an abelian defect group and a cyclic inertial quotient acting freely but not transitively. We prove that when $p=2$, such blocks are inertial, i.e. basic Morita equivalent to their Brauer correspondent. Together with a result of the second author on Singer cycle actions on homocyclic defect groups, this completes the classification of $2$-blocks with a cyclic inertial quotient acting freely on an abelian defect group.


\end{abstract}

\section{Introduction}
Throughout this paper, we work in a $p$-modular system $(K,\cO, k)$ in which $\cO$ is complete discrete valuation ring such that its residue field $k=\cO/J(\cO)$ is algebraically closed of characteristic $p$, and its field of fractions $K$ is considered to be big enough for all the groups considered in this paper. For a finite group $G$, we consider blocks of $\cO G$ (or $kG$), i.e. indecomposable summands of the group algebra. These correspond to decompositions of $1$ into a sum of primitive central idempotents. We study these up to Morita equivalence. Blocks of $\cO G$ and of $kG$ are in a one-to-one correspondence, though this correspondence is not generally known to preserve Morita equivalence class. 

For $B$ a block of $\cO G$, we can define its defect group $D$, the smallest $p$-subgroup of $G$ that contains a vertex of every indecomposable module in $B$. Given a block $e$ of $C_G(D)$ with Brauer correspondent $B$ (denoted $e^G=B$), we define the inertial quotient of $B$ as $E:= N_G(D,e)/DC_G(D)$. This is always a $p'$-group, and it acts faithfully on $D$. We denote the number of ordinary characters in $KG$ that lie in $B$ by $k(B)$, and the number of Brauer characters in $kG$ that lie in $B$ by $l(B)$.

We say that two blocks are Morita equivalent when their module categories are equivalent. In particular, any Morita equivalence between two blocks $A$ and $B$ can be characterised as a tensor product with an $(A,B)$-bimodule $M$ (and its inverse with a $(B,A)$-bimodule) \cite{thevenaz}. We say that $M$ \textit{realises} the Morita equivalence. If $M$ has endopermutation source, then we say that $A$ and $B$ are \textit{basic Morita equivalent}. If $M$ has trivial source, we say that $A$ and $B$ are \textit{source algebra equivalent}. Each equivalence is stronger than the ones before it, and is known to preserve more invariants (for a detailed list, see \cite{ea20}). At the moment, there is no known example of two blocks that are Morita equivalent but not basic Morita equivalent (note that the blocks defined in \cite{dade71} are not counterexamples, because they have the same fusion system), and it is conjectured that there is no such example. 

A block is said to be \textit{inertial} if it is basic Morita equivalent to its Brauer correspondent in $N_G(D)$.

Recall that an action of a group on a set is said to be \textit{free} (or \textit{semiregular}) if, for any two elements in the set, there is at most one group element mapping one to the other. This is equivalent to the statement that the only element of the group that fixes a point is the identity.

Let $D \cong (C_{2^m})^n$, and denote by $\Omega(D)$ its Omega subgroup, generated by the elements of order $2$. Here $\Omega(D) \cong (C_2)^n$ is the unique elementary abelian subgroup of $D$. We can then identify $\Omega(D)$ with the vector space $(\mathbb{F}_2)^n$. Elements of order $2^n-1$ in $\Out(\Omega(D))=\GL_n(2)$ are called Singer cycles, and each generates a \textit{Singer subgroup} which acts transitively on the non-trivial elements of $\Omega(D)$. Singer cycles are well-known in the literature: in \cite[2.1]{mck19} it is demonstrated that they always exist, and that the normaliser of a Singer subgroup in $\GL_n(2)$ is isomorphic to $C_{2^n-1} \rtimes C_n$, where $C_n$ acts on $(\mathbb{F}_2)^n$ as its field automorphisms, generated by the Frobenius automorphism. In \cite{mck19}, the second author showed that if a $2$-block of a finite group has an abelian homocyclic defect group $D$ and an inertial quotient that acts on $\Omega(D)$ as a Singer cycle, then it is either basic Morita equivalent to the principal block of $\SL_2(2^n)$, or it is inertial. 

Singer subgroups are characterised by their action on $\Omega(D)$ being both transitive and free. Here we look at free actions which are not transitive, thereby generalising the results in \cite{mck19}. This includes the so-called \textit{sub-Singer} case in \cite{mckthesis}, in which powers of Singer cycles were studied. We also consider an arbitrary abelian defect group, rather than just homocyclic groups. 


\begin{theorem} \label{maintheorem}
Let $G$ be a finite group, and let $B$ be a block of $\cO G$ (or $kG$) with an abelian defect $2$-group $D$ of rank $n$, and a cyclic inertial quotient $E$ such that $E$ acts freely but not transitively on $\Omega(D) \backslash \{1\}$. Then $B$ is inertial. In particular, $B$ is basic Morita equivalent to $\cO (D \rtimes E)$. Further, $l(B)=|E|$ and $k(B)= |E| + \frac{2^n-1}{|E|}$.
\end{theorem}
 
Note that this result relies on the well-studied properties of quasisimple groups with abelian defect groups for $p=2$, whose blocks were classified in \cite{ekks}. There is no reason to suspect this holds for odd $p$, and indeed in \cite{koshitani03}, there noted is already a counterexample when $p=3$ and $D=(C_3)^2$: the principal block of $\cO A_6$ is not Morita equivalent to its Brauer correspondent, but its inertial quotient $C_4$ acts as a power of a Singer cycle.

We can combine \Cref{maintheorem} and \cite{mck19} to obtain a complete classification of $2$-blocks with cyclic inertial quotients acting freely on abelian defect groups.

\begin{theorem} \label{classifications}
Let $G$ be a finite group, and let $B$ be a block of $\cO G$ (or $kG$) with an abelian defect $2$-group $D$ of rank $n$, and a cyclic inertial quotient $E$ that acts freely on $D$. Then either $B$ is  inertial, or $B$ is basic Morita equivalent to the principal block of $\cO (\SL_2(2^n))$, where the latter can only occur when $D$ is elementary abelian and $E$ acts transitively.
\end{theorem}

\section{Reductions}

We record standard reductions in modular representation theory, which we will use extensively.
\begin{theorem}[Fong's First Reduction \cite{lin18}] \label{fong1}
Let $G$ be a finite group, and let $N \lhd G$. Let $b$ be a block of $\cO N$ and $B$ be a block of $\cO G$ that covers $b$. Let $H=\operatorname{Stab}_G(b)$, where the action of $G$ on $N$ is given by conjugation. Then there is a unique block $C$ of $\cO H$ that covers $b$ and such that, as idempotents, $BC=B$. Further, $C$ is source algebra equivalent to $B$.
\end{theorem}

We also use the following version of Fong's Second Reduction.
\begin{theorem}[\cite{kupu90}] \label{fong2}
Let $G$ be a finite group and $N \lhd G$. Let $B$ be a block of $\cO G$ with defect group $D$ that covers a $G$-stable nilpotent block $b$ of $\cO N$ with defect group $D \cap N$. Then there are finite groups $M \lhd L$ such that $M \simeq D \cap N$, $L/M \simeq G/N$, there is a subgroup $D_L \leq L$ with $D_L \simeq D$ and $M \leq D_L$, and there is a central extension $\tilde{L}$ of $L$ by a $p'$-group, and a block $\tilde{B}$ of $\cO \tilde{L}$ which is basic Morita equivalent to $B$ and has defect group $\tilde{D} \simeq D_L \simeq D$ and the same inertial quotient as $B$.
\end{theorem}
\begin{proof}
Guidance on the extraction of this result from \cite{kupu90} can be found in \cite[2.2]{ea14}. The fact that the Morita equivalence is basic follows from \cite[8.2.18]{lin18}.
\end{proof}
In particular, we use the following immediate corollary, originally stated in \cite{ea17}.
\begin{corollary}[\cite{ea17}] \label{corollaryfong2}
Let $G$ be a finite group, let $N \lhd G$ with $N \not\leq Z(G)O_p(G)$. Let $B$ be a quasiprimitive $p$-block of $\cO G$ covering a nilpotent block $b$ of $\cO N$. Then there is a finite group $H$ with $[H:O_{p'}(Z(H))] < [G:O_{p'}(Z(G))]$ and a block $B_H$ with isomorphic defect group to the one of $B$, such that $B_H$ is basic Morita equivalent to $B$.
\end{corollary}

A block $B$ of $\cO G$ is \emph{nilpotent covered} if these exists a group $\tilde{G}$ and a nilpotent block $\tilde{B}$ of $\cO \tilde{G}$ such that $G \lhd \tilde{G}$ and $\tilde{B}$ covers $B$. This concept is closely related to inertiality, as the following proposition shows.
\begin{lemma}[\cite{pu11}, \cite{zhou16}]  \label{nilcov}
Let $G$ be a finite group and let $N \lhd G$. Let $b$ be a $p$-block of $\cO N$ covered by a block $B$ of $\cO G$. Then 
\begin{enumerate}[(i)]
\item If $B$ is inertial, then $b$ is inertial.
\item If $b$ is nilpotent covered, then $b$ is inertial.
\item If $p$ does not divide $[G:N]$ and $b$ is inertial, then $B$ is inertial.
\item If $b$ is nilpotent covered, then it has abelian inertial quotient.
\end{enumerate}
\end{lemma}
\begin{proof}(i), (ii) are Theorem 3.13 in \cite{pu11}, and (iv) is Corollary 4.3 in the same paper. (iii) is the main theorem of \cite{zhou16}.
\end{proof}



In some cases, the action by conjugation of an overgroup on a block of the group algebra coincides with an inner automorphism of the block. Given $N \lhd G$ and a block $b$ of $N$, the subgroup of all elements that act as an inner automorphism on $b$ is denoted by $G[b]$, and it is a normal subgroup of $G$.

\begin{lemma}[\cite{kekoli}] \label{inneraut}
Let $N \lhd G$, and let $B$ be a block of $G$ covering a block $b$ of $N$. Then there is a block $\hat{b}$ of $G[b]$ that covers $b$ and is covered by $B$, $\hat{b}$ is source algebra equivalent to $b$ and $B$ is the unique block of $G$ that covers $\hat{b}$. 
\end{lemma}

We use two results from \cite{mck19} that relate the inertial quotient of a block of $G$ with those of normal subgroups of $G$. Since these play a crucial role in our proof, we include them here for convenience. 

\begin{lemma}[4.11 \cite{mck19}] \label{iqlemma1}
Let $N \lhd G$ be finite groups. Let $B$ be a block of $G$ with inertial quotient $E_B$ covering a $G$-stable block $b$ of $N$ with inertial quotient $E_b$, such that $B$ and $b$ share an abelian defect group $D$. Let $b_D$ be a Brauer correspondent of $b$ in $C_N(D)$. Then the following statements are true: 
\begin{enumerate} 
		\item \label{IQ tool case1}If $C_G(D)=C_N(D)$, then there is a monomorphism $E_b \hookrightarrow E_B$ whose image is normal. 
		\item \label{IQ tool case2} If $C_N(D) \neq C_G(D)$ and $C_G(D)_{b_D}=C_N(D)$, then there is a monomorphism $E_b \hookrightarrow E_B$ whose image is normal.
		\item \label{IQ tool case3} If $[C_G(D):C_N(D)]=[G:N]$ and $C_G(D)_{b_D}=C_G(D)$, then there is a monomorphism $E_B \hookrightarrow E_b$. 
\end{enumerate} 
In particular, if $[G:N]$ is prime, then either $E_b$ is isomorphic to a normal subgroup of $E_B$, or $E_B$ is isomorphic to a subgroup of $E_b$. 	
\end{lemma}

\begin{lemma}[4.12 \cite{mck19}] \label{iqlemma2}
	Let $N \lhd G$ be finite groups such that $G/N$ is abelian, and let $B$ be a block of $G$ with inertial quotient $E_B$ covering a block $b$ of $N$ with inertial quotient $E_b$. Suppose that $B$ and $b$ share an abelian defect group $D$. Then there is a chain $N \lhd K \lhd G$ and a block $b_K$ of $K$ with inertial quotient $E_{b_K}$, and there are monomorphisms $E_b \hookleftarrow E_{b_K} \hookrightarrow E_B$ such that the image of $E_{b_K}$ is normal in $E_B$. 
\end{lemma}

We will consider groups $N \lhd G$ where the action of $G/N$ normalises but does not centralise $D$, and hence induces non-trivial automorphisms of $D$. In this situation we use the following lemma, originally stated for split extensions in \cite{mckthesis}.

\begin{lemma}[4.3.7 \cite{mckthesis}] \label{stack}
	Let $L \lhd G$, and let $D \leq L$. Then $C_G(D)=C_L(D)$ if and only if each nontrivial element of $N_G(D) \setminus N_L(D)$ induces an automorphism of $D$ that is not in the image of $N_L(D)/C_L(D)$ in $\aut(D)$. 

\end{lemma}
%
%
%

The $2$-blocks of quasisimple groups with abelian defect groups are classified in \cite{ekks}, the main result of which is as follows:

\begin{proposition}[6.1 \cite{ekks}] \label{quasisimpleblocks}
	
Let $L$ be a quasisimple group, and $b$ be a $2$-block of $\mathcal{O}L$ with abelian defect group $D$. Then one of the following holds. 
\begin{enumerate}			
		\item \label{EKKSi} $L/Z(L)$ is one of ${}^2G_2(q)$, $J_1$, or $SL_2(2^a)$, and $b$ is the principal block. 
		
		\item \label{EKKSii} $L$ is $Co_3$, $D \cong (C_2)^3$, and $b$ is the unique non-principal block of $L$.
		
		\item \label{EKKSiii} $b$ is nilpotent-covered, i.e. there exists a finite $\tilde{L} \rhd L$ with $Z(\tilde{L}) \geq Z(L)$ such that a nilpotent block of $\tilde{L}$ covers $b$.
		
		\item \label{EKKSiv} For some $M=M_0 \times M_1 \leq L$, $b$ is Morita equivalent to a block $b_M$ of $\mathcal{O}M$ such that the defect groups of $b_M$ are isomorphic to $D$. Further, $M_0$ is abelian, and the block of $\mathcal{O}M_1$ covered by $b_M$ has Klein Four defect groups. In particular, $b$ is Morita equivalent to a tensor product of a nilpotent block and a block with defect group $C_2 \times C_2$.	\end{enumerate}
\end{proposition}
\begin{proof}Guidance on the extraction from \cite{ekks} of the additional details in the statement of this theorem compared to the form in \cite{ekks} is given in \cite[2.9]{ar19} (see also \cite{ea17}).\end{proof}

\section{Free actions and blocks}
Considering a block whose defect group $D$ is abelian of rank $n$, we can often characterise the action of the inertial quotient by its action on $\Omega(D) \cong (C_2)^n$.  When the inertial quotient is generated by a power of a Singer cycle this action is always free, though the converse is not true: a counterexample is given when $D=(C_2)^6$ and $E=C_7$, labeled as (x) in \cite{ar20}. Nevertheless, when a cyclic group acts freely on $\Omega(D)$, the action is remarkably similar to that of an inertial quotient generated by a power a Singer cycle (called a ``sub-Singer action'' in \cite{mckthesis}, where this case is studied extensively). Here we explore the properties of free actions:

\begin{lemma} \label{orbitstab}
Let $D$ be a finite abelian $2$-group of rank $n$, and let $E \leq \Out(D)$. Suppose $E$ acts freely on the non-trivial elements of $\Omega(D)$. Then $E$ acts freely on the non-trivial elements of $D$, and $|E|$ divides $2^n -1$.
\end{lemma}
\begin{proof}
Suppose that $E$ does not act freely on the nontrivial elements of $D$. Then there are non-trivial elements $d \in D$ and $e \in E$ such that $d^e = d$. For some $k \in \mathbb{N}$, $1 \neq d^k \in \Omega(D)$, and hence we have $(d^k)^e=d^k$. This is a contradiction, since $E$ acts freely on $\Omega(D) \backslash \{1\}$: hence, $E$ must act freely on $D \backslash \{1\}$.

Since the stabiliser of any nontrivial element in $D$ is trivial, the orbit-stabiliser theorem tells us that each orbit has size $|E|$. Each element occurs in exactly one orbit, and so there are $(2^n-1)/|E|$ orbits. In particular, $|E|$ divides $2^n-1$.
\end{proof}

Next we consider the structure of the orbits in $D$ under the action of $E$:
\begin{lemma}[\cite{mckthesis}] \label{singo}
Let $x$ be a Singer cycle of $\mathbb{F}_{p^n}$, and $s$ a proper divisor of $p^n-1$. Let $O$ be an orbit of a nontrivial element. Then $O \cup \{0\}$ is a subgroup of $\mathbb{F}_{p^n}^\times$ if and only if $|x^s|=p^r-1$ for some $r$. 
\end{lemma}

Note that $r$ divides $n$ if and only if $p^r-1$ divides $p^n-1$. 

\begin{proposition} \label{sub DcapN}
Let $B$ be a block of $\cO G$ with an abelian defect $2$-group of rank $n$ and a cyclic inertial quotient $E$ acting freely on $\Omega(D) \backslash \{1\}$. Let $N \lhd G$. Then there are integers $r,s$ with $r>1$ and $n=rs$, and a decomposition of $D \cap N$ into $s$ homocyclic direct factors of rank $r$: $$D \cap N \cong (C_{p^{m_1}})^r \times (C_{p^{m_2}})^r \times \dots \times (C_{p^{m_s}})^r $$ where $0 \leq m_i \leq \log_2(|D|)$. Further, if $|E|$ divides $p^r-1$ for some $r<n$, then $\Omega(D) \cap N \cong (C_p)^{ra}$ where $1 \leq a \leq n/r$. Otherwise, $\Omega(D) \leq N$. 
\end{proposition}
\begin{proof}
By \Cref{orbitstab}, $E$ acts freely on $D \backslash \{1\}$. Let $\mathcal{Q}=\{q_1, \dots, q_n\}$ be a minimal generating set of $D$, and let $\mathcal{D}=\{d_1, \dots, d_n\}$ be the corresponding elements in $\Omega(D)$ (also a generating set for $\Omega(D)$). Let $O_{d_1}$ denote the orbit of $d_1$ under $E$, and let $X=\{d_{x_1}, \dots, d_{x_y}\}$ be a minimal subset of $\mathcal{D}$ such that $O_{d_1} \leq \langle X \rangle$. If $\Omega(D) \cap N$ contains $\langle X \rangle$, then there are integers $m_1, \dots, m_s$ it also contains $\langle \tilde{X} \rangle$, where $\tilde{X} = \{q_{x_1}^{m_1}, \dots, q_{x_y}^{m_y}\}$. Thus, $D \cap N$ decomposes into $s$ homocyclic direct factors as stated above (some of which are possibly trivial). 

If $|E|$ does not divide $p^r-1$, by Lemma \ref{orbitstab} the only possibility is that $X=\mathcal{D}$. Otherwise, it may happen that $\Omega(D) \cap N \neq N$, in which case still Lemma \ref{singo} applied to $\langle X \rangle$ proves the statement.


 
 
 \end{proof}

\section{Proof of the main theorem}

Now we are ready to prove \Cref{maintheorem}:

\begin{proof}
Let $G$ be a finite group, and $B$ a $2$-block of $\mathcal{O}G$ with abelian defect group $D$ of rank $n$, and cyclic inertial quotient $E$ such that $E$ acts freely but not transitively on $\Omega(D) \backslash \{1\}$. Note that \Cref{orbitstab} tells us that $E$ acts freely on the non-trivial elements of $D$. 

We define an ordering on blocks of finite groups which satisfy these conditions by considering the lexicographic ordering on $(n, |E|, [G:O_{2'}(Z(G))], |G|)$ (recall that $|E|$ divides $2^n-1$ by \Cref{orbitstab}). Our strategy is to consider a minimal counterexample to \Cref{maintheorem}. That is, we consider the smallest block with respect to this ordering, that satisfies our assumptions but is not inertial. From this we will derive a contradiction. Note that, by the classification in \cite{eali18b}, $n > 3$. 

 We record three important consequences of the assumptions above:
\begin{itemize}
\item[(I)] $B$ is quasiprimitive. That is, for any $N \lhd G$ the blocks of $\cO N$ covered by $B$ are $G$-stable: \Cref{fong1} tells us that $\operatorname{Stab}_G(b)$ has a block $\hat{b}$ basic Morita equivalent to $B$, and so by minimality we have $\operatorname{Stab}_G(b)=G$, implying $b$ is $G$-stable.
\item[(II)] If $N\lhd G$ is such that $B$ covers a nilpotent block $b$ of $N$, then $N \leq O_2(G)Z(G)$. This follows from \Cref{fong2} and \Cref{corollaryfong2}.
\item[(III)] Any block with defect group $D$ and cyclic inertial quotient $F$ acting freely on $D$ with $|F|<|E|$ is inertial. This immediately follows from minimality.
\end{itemize}

Recall that a component of $G$ is a subnormal quasisimple subgroup. The product of all components forms a characteristic subgroup of $G$, called the layer and usually denoted as $E(G)$. Further, the layer is always isomorphic to a central product of the components. Let $F(G)$ be the Fitting subgroup of $G$. Then $F^*(G):=F(G)E(G)$ is the generalised Fitting subgroup of $G$ \cite{gor07}. This is a characteristic subgroup of $G$, and is a central product of the layer and Fitting subgroup. Further, $F^*(G)$ is self-centralising. That is, $C_G(F^*(G)) \leq F^*(G)$.

From (II), we can assume that $F(G)=O_2(G)Z(G)$ since any block of a normal subgroup with odd order has trivial defect. 

Let $C:=C_G(O_2(G)) \lhd G$, and let $c$ be the unique block of $\cO C$ covered by $B$, and $E_c$ its inertial quotient. Then, since $C_G(D) \leq C$, \Cref{iqlemma1} tells us that $E_c \leq E$. Since $E_c$ centralises $O_2(G) \leq D$, we have either $O_2(G)=1$ or $E_c=1$. In the latter case, $c$ is nilpotent and hence $C \leq O_2(G)Z(G)$, in which case $D=O_2(G)$ and the main theorem of \cite{kul85} implies that $B$ is inertial. This is a contradiction to our assumption that $B$ is a counterexample, and so may assume that $O_2(G)=1$. As a corollary, $Z(E(G))$ must have odd order. Further, from \cite[7.6]{sam14} we deduce that $Z(E(G))$ is fixed by the action of $G/F^*(G)$, and hence that $Z(E(G)) \leq Z(G)$.

Consider the layer as a central product of the components, which we denote by $L_i$: $$E(G)=L_1 * L_2 * \dots * L_t.$$  Let $b_E$ be the block of $E(G)$ covered by $B$, and $b_i$ a block of $\cO L_i$ covered by $b_E$. The action of $G$ permutes the components, and so we may consider the $G$-orbit of component. Letting $W$ denote the normal subgroup generated by such an orbit, and letting $b_W$ denote a block of $\mathcal{O}W$ covered by $B$, we may assume that $b_W$ is not nilpotent due to (II). Consequently, no block $b_i$ of any component within that orbit can be nilpotent, since each component in the orbit is isomorphic and so their blocks have isomorphic defect groups. In particular, $b_i$ is not nilpotent for any $i$. This implies that $D \cap L_i$ has rank at least $2$, since for $p=2$ a block with cyclic or trivial defect group is always nilpotent. Since $Z(E(G))$ has odd order, \Cref{sub DcapN} implies that $$D \cap F^*(G) = \prod_{i=1}^t (D \cap L_i)$$

Let $\sigma: G \to \Sigma_t$ be the group homomorphism to the symmetric group on $t$ elements induced by the permutation action of $G$ on the set of its components. Let $K:=\ker(\sigma)$, and let $B_K$ be the unique block of $\cO K$ covered by $B$. Note that $D \leq K$, as every permutation of the components induces a permutation of the elements in the defect group. Further, note that $C_K(D)=C_G(D)$, since an element permutes the set $\{L_i\}$ if and only if it permutes the set $\{D \cap L_i\}$. Then, in particular, we have $C_G(D) \leq K$, and so by \cite[15.1.5]{alp151} $B$ is the unique block of $\cO G$ covering $B_K$. By \cite[15.1.4]{alp151}, this implies $G/K$ has odd order, and therefore it is solvable.

From \cite{bidwell}, there exists an embedding $$K/F^*(G) \hookrightarrow \prod\limits_{i=1}^t \Out(L_i).$$ The latter group is solvable due to Schreier's conjecture, and in fact by considering  \cite[Table 5]{atlas}, it must be supersolvable, implying $K/F^*(G)$ is supersolvable. By applying \cite[3.4]{ar19}, $DF^*(G)/F^*(G)$ is therefore a Sylow $2$-subgroup of $K/F^*(G)$. Suppose that $D \not\leq F^*(G)$, and take a chain of normal subgroups of $G$ as follows, with prime indices, and blocks $b_i$ of $\cO N_i$ covered by $B$ with defect groups $D_i := D \cap N_i$:
$$F^*(G) = N_0 \lhd N_1 \lhd \dots \lhd N_s = K.$$
Let $j$ be the smallest integer such that $[N_{j+1}:N_j]=2$. By \cite[6.8.11]{lin18} $b_{j+1}$ is the unique block that covers $b_j$, and hence by \cite[15.1.4]{alp151} we know $[D_{j+1} : D_j]=2$. If $D_j$ has a complement $R$ in $D_{j+1}$, then $R=C_2$ is fixed by the action of $E$, which is a contradiction to $E$ acting freely on $D \backslash \{1\}$. Otherwise, consider decompositions of $D_{j+1}$ and $D_j$ as in \Cref{sub DcapN}: both defect groups decompose as products of homocyclic direct factors of rank $r$. However, at most one cyclic factor can change, which forces $r=1$, contradicting the fact that $r>1$. Hence, $D \leq F^*(G)$, and so $K/F^*(G)$ has odd order.

Since $C_K(D)=C_G(D)$, we may recursively apply Lemma \ref{iqlemma1} to a chain of normal subgroups, telling us that $E_K \leq E$, and so $E_K$ is cyclic. If $E_K$ is a proper subgroup of $ E$, then by induction $B_K$ is inertial, and hence by \Cref{nilcov} so is $B$. Hence, we may assume that $E_K=E$, and so by minimality we must have $K=G$, since $O_{2'}(Z(G)) \leq K$ and hence $[K:O_{2'}(Z(K))] \leq [G:O_{2'}(Z(G))]$. In particular, we've shown that each component $L_i$ is normal in $G$.

Next, we will show there are few possibilities for the structure of each component, and each has inertial blocks. Without loss of generality, it is enough to prove that a single component, namely $L_1$, must belong to that set of possibilities.

Note that since all components are normal in $G$, by \cite{bidwell} we can assume that $G$ is a subgroup of a perfect central extension $H$ of $\prod_{i=1}^t \aut(L_i/Z(L_i))$, and further that $G$ is normal in this group, since $G/F^*(G)$ has odd order and all outer automorphism group are solvable by Schreier's conjecture, so $G/F^*(G)$ is supersolvable. In other words, $H$ is a central product of almost quasisimple groups, and it serves as the ambient group for all the extensions that we will consider.

We record a construction that we use in several cases: when $L_1$ is isomorphic to a simple group $S$, we can write $E(G)=S \times T$. Let $M=SC_G(S) \lhd G$. Since $S$ is normal in $G$, so is $C_G(S)$. Since $C_G(S) \cap S = \{1\}$, we have $M =  S \times C_G(S)$. Now, let $b_M$ be the unique block of $\cO M$ covered by $B$: it is the tensor product of $b_1$ and a uniquely determined block $b_C$ of $\cO C_G(S)$, so in particular its inertial quotient has the form $E_M \cong E_1 \times F$. Further, $C_{G/F^*(G)}(S) \leq \Out(T)$, so in particular $[G:M]$ is at most $|\Out(S)|$. 

In this framework, we examine each possibility from Theorem \ref{quasisimpleblocks} to determine which ones can occur.

\begin{itemize}
\item If $L_1=J_1$ or $L_1=\operatorname{Co}_3$ then in particular $L_1$ is simple. Further, $\Out(L_1)=1$, so $M=G$, and hence the group $G$ has $S$ as a direct factor. But then $E = E_S \times F$ where $E_S=C_7 \rtimes C_3$ (\cite{ea14}). This cannot happen because $E$ is cyclic, a contradiction.

\item If $L_1 = S = \SL_2(2^a)$, since $\Out(S)\cong C_a$ then $G/M$ can be embedded in $C_a$ and is, hence, cyclic and a fortiori abelian. By Lemma \ref{iqlemma2}, there exists a chain $M \lhd J \lhd G$ with blocks $b_M, b_J, B$ such that the inertial quotients satisfy the following: $E_M \geq E_J \leq E$. By (III) $E_J=E$ and hence, by minimality, $J=G$.

By construction, $E_M=C_{2^a-1} \times F$ for a finite group with odd order $F$. Suppose that $F=1$. Then the block of $C_G(S)$ covered by $B$ is nilpotent, so from (II) $C_G(S) \leq O_2(G)Z(G)$. In particular, $E(G)=S$. But then $G \leq \Aut(S)$ and, therefore $E$ acts transitively on $D$, a contradiction. Hence, we assume that $F \neq 1$. 

We claim that $M=J$. This is implied by $C_J(D)=C_M(D)$, as then all intermediate centralisers between $M$ and $J$ do not increase and we can use Lemma \ref{iqlemma1} for each inclusion of normal subgroups, compose the embeddings and get an embedding $E_M \leq E_J$, which since $E_J \leq E_M$ implies that $E_M = E_J$, and hence $M=J$.

Let $x$ be a generator of $G/M$. By Lemma \ref{stack}, our claim is equivalent to the following statement:
$$\forall x \in N_G(D) \setminus N_M(D) \quad , \quad \forall m \in N_M(D) \quad, \quad \tau_x|_D \neq \tau_m |_D$$
where $\tau_x$, $\tau_m$ denote the automorphisms of $D$ induced by conjugation.

Let $\Out_G(M)$ be the image in $\Out(M)$ of the map $g \mapsto \tau_g$. Since $M$ is a direct product of two subgroups that are normal in $G$, $\Out_G(M)=\Out_G(S) \times \Out_G(C_G(S))$. In particular, by taking preimages for each fixed $x$ as above there are decompositions $\tau_x=\tau_x^S\tau_x^C$ and $\tau_m = \tau^S_m \tau^C_m$. 

Suppose that $C_J(D) \neq C_M(D)$. By Lemma \ref{stack} then there exists $m \in N_M(D)$ such that $\tau_m |_D = \tau_x|_D$. Then $\tau_m^S = \tau_x^S$, and $\tau_x^S$ acts on $D \cap S$ as a field automorphism of $S$. 

Being an element of a direct product, $m$ decomposes as $m = st = ts$, for $s \in N_S(D \cap S)$ and $t \in T$, where $\tau_t |_S = \operatorname{Id}_S$, and $\tau_m^S|_D = \tau_s^S|_D$. Hence, there is an element $s \in N_{\SL_2(2^a)}(D)$ whose action by conjugation coincides with the one of a nontrivial power of a field automorphism of $\SL_2(2^a)$. Since $N_{\SL_2(2^a)}(D) = D \rtimes C_{2^a-1}$, and $C_{2^a-1}$ acts as a Singer cycle on $D$ (see, for example, \cite{mck19}). However, the action of $s$ on $D$ fixes one nontrivial element by \cite[2.5]{mck19}, which is a contradiction.

Then $C_J(D)=C_M(D)$ and hence $M=J$ and $J=G$, so $M=G$. But then $E_M = E$. However, $E_M$ is not cyclic: a contradiction.

\item If $L_1$ is of type $D_n(q)$ or $E_7(q)$, suppose without loss of generality that $D \cap L_1 \geq (C_2)^3$ (as otherwise it can be included in the Klein four case, that we examine later). First, note that we can assume that $L_1=S$ is a simple group, as (again from \cite[Table 5]{atlas}) the Schur multipliers have even order but $Z(E(G))$ has odd order. Again from \cite[Table 5]{atlas}, $G/M$ is cyclic since it embeds in $\Out(S)$. We define $M$ as above, and note that $E_M \geq E$, and $E_M = C_3 \times F$. Now $D \cap S = (C_2)^3$ contains an element which is fixed both by the action of $C_3$ and by the action of $F$. In particular, then, every subgroup of $E_M$ fixes it. This is a contradiction to $E$ acting freely on $D$.

\item If $L_1$ is such that $D \cap L_1$ is a Klein four group, then by \cite{erd82, lin18} $b_1$ is source algebra equivalent to either $\cO A_4$ or $B_0(\cO A_5)$. In particular, the orbit $D \cap L_1$ has size $3$, so from Lemma \ref{sub DcapN} $E \cong C_3$ and every orbit has size $3$. If $b_1$ is source algebra equivalent to $\cO A_4$, then in particular it is inertial. 

For a block $c$ there is an injective map $\Out(c) \to \Pic(c)$ \cite[55.11]{cure87}, and since by \cite{eali18} we have that $\Pic(B_0(\cO A_5)) = C_2$, we deduce that $G/L_1$ acts as an inner automorphism on $b_1$.

Similarly, if $L_1 = {}^2 G_2(q)$ from \cite{g2inner} every automorphism of $L_1$ induces an inner automorphism of $b_1$. Hence, suppose that either $L_1 = {}^2G_2(q)$ or $b_1$ is Morita equivalent to $B_0(\cO A_5)$. Either way, $G/L_1$ induces an inner automorphism of $b_1$.

We repeat the construction of $M=L_1C_G(L_1)$, except that $L_1 \cap C_G(L_1) = Z(L_1)$, so in general $M = L_1 * C_G(L_1)$ is just a central product (not necessarily direct). Still, from \cite[7.5]{sam14} $b_M$ is isomorphic to $b_1 \otimes b_C$ where $b_1$ is a block of $\cO L_1$. In particular, they share an inertial quotient. Note that by \cite[7.6]{sam14} since all components are normal $G/M$ can be embedded in $\Out(L_1) \times \Out(T)$.

Let $Q\cong G/M$, and let $\phi_1 \in \aut(L_1)$, $\phi_C \in \aut(C_G(L_1))$ such that a preimage of a generator of $Q$ in $G$ acts on $M$ as $\phi_1\phi_C$ (so $\phi_1$ induces an inner automorphism of $b_1$).

Consider the extension $N=(L_1.Q) * (C_G(L_1).Q)$ where $L_1.Q \cong M.Q/C_G(L_1)$ (still inducing $\phi_1$), and $C_G(L_1).Q \cong M.Q/L_1$ (still inducing $\phi_C$). Consider blocks of each group linked by the covering relation and their inertial quotients, as in the following diagram:

\begin{minipage}{\linewidth}

\begin{displaymath}
\xymatrix@R=9pt@C=3pt{
&   {\begin{array}{cc} H =(L_1.Q) * C_G(L_1) \\ b_H \cong b_{1Q} \otimes b_C \\ E_1 \times F \end{array}} \ar[dr]^{\lhd} & \\
{\begin{array}{cc} M=  L_1*C_G(L_1) \\ b_M \cong b_1 \otimes b_C \\  E_1 \times F \end{array}} \ar[ur]^{\lhd}_{SA} \ar[dr]^{\lhd} &  &  {\begin{array}{cc} N=  (L_1.Q) * (C_G(L_1).Q) \\ B_N \cong b_{1Q} \otimes b_{CQ} \\ E_N = E_1 \times F' \end{array}} \\
&    {\begin{array}{cc} G=  (L_1*C_G(L_1)).Q \\ B \\ E\end{array}}  \ar[ur]^{\lhd}_{SA} &
}
\end{displaymath}
\end{minipage}

By \Cref{inneraut}, $b_1$ is source algebra equivalent to $b_{1Q}$ and therefore $b_H$ is source algebra equivalent to $b_M$. Similarly, $B$ is source algebra equivalent to $B_N$ because $N[B] = N$, since the action on $M$ of $N/M.Q \cong Q$ differs from the action of $M.Q/M \cong Q$ just by $\phi_1$, which is induced by an inner automorphism of $b_1$.

The inertial quotient of $B_N$ is $E_N = E_1 \times F'$ for some (possibly trivial) group $F'$, which can be deduce by considering the normal subgroup $H$ as an intermediate step.

If $L_1={}^2G_2(q)$ then $E_1 = C_7 \rtimes C_3$, a contradiction since $E_N \cong E$ and $E$ is cyclic. Otherwise, if $b_1$ is Morita equivalent to $B_0(\cO A_5)$ then $D \cap L_1 = (C_2)^2$ contains a nontrivial orbit of $E$, so $E\cong C_3$ and hence, since $E_1 = C_3$ already, $F'=1$. But then, since $C_D(E)=1$, $D=[D,E]= D \cap L_1$, a contradiction since then $E$ acts transitively on $D$.

\item Finally, if the block $b_1$ of $\cO L_1$ is nilpotent covered, then it is inertial by \Cref{nilcov}.
\end{itemize}

We have shown that each component $L$ is such that the block of $\cO L$ covered by $B$ is inertial, and hence the block of $E(G)$ covered by $B$ is also inertial. By Theorem \ref{nilcov}, since $[G:E(G)]$ is odd, then $B$ is inertial. 

Hence, by the main theorem of \cite{kul85}, $B$ is basic Morita equivalent to a twisted group algebra $\cO_\alpha(D \rtimes E)$, where $\alpha \in H^2 (E, k^\times)$. Since $E$ is cyclic, $H^2(E, k^\times)=1$, and hence $\alpha=1$ and $B$ is Morita equivalent to the group algebra $\cO(D \rtimes E)$. Then by Lemma 3.8 in \cite{ar20}, we must have $l(B)=|E|$. 

Recall that a $B$-subsection $(u,b_u)$ is defined as a Brauer $B$-subpair $(\langle u\rangle ,b_u)$. Since $B$ is basic Morita equivalent to $\cO(D \rtimes E)$, for its fusion system it holds that $\mathcal{F}(B) \cong \mathcal{F}(D \rtimes E)$. In particular, to compute invariants of subsections it is enough to do so in a set of representatives $\mathcal{R}$ of the orbits of $D$ under the action of $E$. From Brauer’s second main theorem

$$k(B)-l(B) = \sum_{(u,b_u) \in \mathcal{R}, u \neq 1} l(b_u)$$

Since $E$ acts freely, $b_u$ is a nilpotent block for each $u \neq 1$, so in particular $l(b_u)=1$. Then $k(B)-l(B) = |\mathcal{R}|$. By the orbit-stabiliser theorem, $|\mathcal{R}| = \frac{2^n-1}{|E|}$. Since $l(B)=|E|$, we are done.
\end{proof}

\small \

\end{document}